\documentclass[a4paper,12pt]{amsart}

\usepackage{fullpage}

\usepackage[utf8]{inputenc}
\usepackage{amssymb}
\usepackage{amsmath}
\usepackage{enumerate}

\usepackage{color}

\newcommand{\N}{\mathbb{N}}

\newcommand{\Z}{\mathbb{Z}}
\newcommand{\Q}{\mathbb{Q}}

\newcommand{\mH}{\mathcal{H}}
\newcommand{\mS}{\mathcal{S}}
\newcommand{\mP}{\mathcal{P}}

\newcommand{\IS}{I_{\mS}}
\newcommand{\lS}{\leq_{\mS}}
\newcommand{\x}{\mathbf{x}}
\renewcommand{\t}{\mathbf{t}}

\newcommand{\mG}{\mathcal{G}}
\renewcommand{\le}{\leqslant}
\renewcommand{\leq}{\leqslant}
\renewcommand{\ge}{\geqslant}
\renewcommand{\geq}{\geqslant}

\theoremstyle{plain}
\newtheorem{theorem}{Theorem}[section]
\newtheorem{lemma}[theorem]{Lemma}
\newtheorem{corollary}[theorem]{Corollary}
\newtheorem{proposition}[theorem]{Proposition}

\theoremstyle{definition}

\newtheorem{example}[theorem]{Example}

\theoremstyle{remark}

\newtheorem{case}{Case}

\title{M\"obius function of semigroup posets through Hilbert series}
\author[J. Chappelon]{Jonathan Chappelon\,*}
\thanks{* Corresponding Author: Phone/Fax: +33-467144166. Email: jonathan.chappelon@um2.fr}
\address{Universit\'{e} de Montpellier, Institut Montpelli\'{e}rain Alexander Grothendieck, Case Courrier 051, Place Eug\`{e}ne Bataillon, 34095 Montpellier Cedex 05, France}
\email{jonathan.chappelon@umontpellier.fr}
\author[I. Garc\'{i}a-Marco]{Ignacio Garc\'{i}a-Marco}
\email{ignacio.garcia-marco@umontpellier.fr}
\author[L.P. Montejano]{Luis Pedro Montejano}
\email{lpmontejano@gmail.com}
\author[J.L. Ram\'{i}rez Alfons\'{i}n]{Jorge Luis Ram\'{i}rez Alfons\'{i}n}
\email{jramirez@umontpellier.fr}
\date{July 14, 2015}

\begin{document}

\begin{abstract}
In this paper, we investigate the M\"obius function $\mu_{\mathcal{S}}$ associated to a (locally finite) poset arising from a semigroup $\mathcal{S}$ of $\mathbb{Z}^m$. We introduce and develop a new approach to study $\mu_{\mathcal{S}}$ by using the Hilbert series of $\mathcal{S}$. The latter enables us to provide formulas for $\mu_{\mathcal{S}}$ when $\mathcal{S}$ belongs to certain families of semigroups. Finally, a characterization for a locally finite poset to be isomorphic to a semigroup poset is given.
\end{abstract}

\maketitle

\section{Introduction}\label{introduction}

The {\it M\"obius function} is an important concept that was introduced by Gian-Carlo Rota more than 50 years ago in \cite{Rota}. It is a generalization to ({\it locally finite}) posets of the classical M\"obius arithmetic function on the integers (given by the M\"obius function of the poset obtained from the positive integers partially ordered by divisibility). We refer the reader to \cite{Rota} for a large number of its applications.

In this paper, we investigate the M\"obius function associated to posets arising naturally from subsemigroups of $\Z^m$ as follows. Let $a_1, \ldots,a_n$ be nonzero vectors in $\Z^m$ and let $\mS = \langle a_1,\ldots,a_n \rangle$ denote the semigroup generated by $a_1, \ldots,a_n$, that is,
$$
\mS = \langle a_1,\ldots,a_n \rangle = \{x_1 a_1 + \cdots + x_n a_n \, \vert \, x_1,\ldots,x_n \in \N\}.
$$
We say that $\mS$ is {\it pointed} if $\mS \cap (-\mS) = \{0\}$, where $-\mS := \{- x \, \vert \, x \in \mS\}$. Whenever $\mS$ is pointed, $\mS$ induces on $\Z^m$ a poset structure whose partial order $\lS$ is defined by $x \lS y$ if and only if $y - x \in \mS$ for all $x$ and $y$ in $\Z^m$. This (locally finite) poset will be denoted by $(\Z^m, \leq_{\mS})$. We denote by $\mu_{\mS}$ the M\"obius function associated to $(\Z^m,\lS)$. As far as we are aware, $\mu_{\mS}$ has only been investigated when $\mS$ is a {\it numerical semigroup}, i.e., when $\mS \subset \N$ and $\gcd\{a_1,\ldots,a_n\} = 1$. Moreover, the only known results concerning $\mu_{\mS}$ are an old theorem due to Deddens \cite{Deddens}, which determines the value of $\mu_{\mS}$ when $\mS$ has exactly two generators, and a recent paper due to Chappelon and Ram\'{i}rez Alfons\'{i}n \cite{CA}, where the authors investigate $\mu_{\mS}$ when $\mS = \langle a, a + d, \ldots, a + kd \rangle$ with $a, k , d \in \Z^+$. In both papers, the authors approach the problem by a thorough study of the intrinsic properties of each semigroup. Here, we introduce and develop a new and more general method to study $\mu_{\mS}$ by means of the {\it Hilbert series of the semigroup $\mS$}. This enables us to provide formulas for $\mu_{\mS}$ when $\mS$ belongs to some families of semigroups. We also investigate when a locally finite poset is isomorphic to a semigroup poset.

This paper is organized as follows. In the next section, after  reviewing some standard notions of the M\"obius function, we then interpret them for semigroup posets. In Section~3, we present two general results (Theorems \ref{chido} and \ref{genfunmob}) giving a new and general approach to study $\mu_{\mS}$ through the \emph{Hilbert series of the semigroup $\mS$}. This enables us in Section~4 to provide formulas for $\mu_{\mS}$ when $\mS$ is a \emph{semigroup with a unique Betti element} and when $\mS = \langle a_1,a_2,a_3 \rangle \subset \N$ is a {\it complete intersection numerical semigroup} (generalizing results in \cite{CA, Deddens}). Finally, in Section~5, we characterize those locally finite posets $\mP$ that are isomorphic to the poset associated to a semigroup $\mS$. In this case $\mu_{\mP}$ can be computed by means of $\mu_{\mS}$ (this will be illustrated with the well-known classical M\"{o}bius arithmetic function).

\section{M\"{o}bius function associated to a semigroup poset}\label{basic}

Let $(\mP,\leq_{\mP})$ be a partially ordered set, or {\it poset} for short. The {\it strict partial order} $<_{\mP}$ is the reduction of $\leq_{\mP}$ given by $a <_{\mP} b$ if and only if $a\leq_{\mP} b$ and $a\neq b$. Let $a$ and $b$ be two elements of the poset $\mP$. The \emph{interval} between $a$ and $b$ is defined by
$$
{\left[a,b\right]}_{\mP} := \left\{ c\in \mP\ \middle|\ a\leq_{\mP} c\leq_{\mP} b \right\}.
$$
A poset is said to be {\it locally finite} if every interval has finite cardinality. We only consider locally finite posets in this paper. A {\it chain} of length $l\geq 0$ between $a$ and $b$ is a subset of ${\left[a,b\right]}_{\mP}$ containing $a$ and $b$, of cardinality $l+1$ and totally ordered by $<_{\mP}$, that is $\left\{ a_0,a_1,\ldots,a_l\right\}\subset{\left[a,b\right]}_{\mP}$ such that
$$
a=a_0 <_{\mP} a_1 <_{\mP} a_2 <_{\mP} \cdots <_{\mP} a_{l-1} <_{\mP} a_l = b.
$$
For any nonnegative integer $l$, we denote by $c_l(a,b)$ the number of distinct chains between $a$ and $b$ of length $l$. This number always exists because the poset $\mP$ is assumed to be locally finite.

For instance, the number of chains $c_2(2,12)$, where the poset is $\N$ partially ordered by divisibility, is equal to $2$. Indeed, there are exactly $2$ chains of length $2$ between $2$ and $12$ in ${\left[2,12\right]}_{\N}=\left\{2,4,6,12\right\}$, which are $\left\{2,4,12\right\}$ and $\left\{2,6,12\right\}$.

For any locally finite poset $\mP$, the {\it M\"{o}bius function} $\mu_{\mP}$ is the integer-valued function on $\mP\times \mP$ defined by
\begin{equation}\label{eq1}
\mu_{\mP}(a,b) = \sum_{l\geq 0}{{(-1)}^{l}c_l(a,b)},
\end{equation}
for all elements $a$ and $b$ of the poset $\mP$. Note that this sum is always finite because, for $a$ and $b$ given, the interval ${\left[a,b\right]}_{\mP}$ has finite cardinality. The concept of M\"{o}bius function for a locally finite poset $(\mP,\le)$ was introduced by Rota in \cite{Rota}. There, Rota proves the following property of the M\"obius function: for all $(a,b)\in \mP\times \mP$,
\begin{equation}\label{eq2}
\mu_{\mP}(a,a) = 1 \quad \text{and} \quad \sum_{c\in
\left[a,b\right]_{\mP}}\mu_{\mP}(a,c) = 0 \text{, if} \ a <_{\mP} b.
\end{equation}

Here, posets associated to semigroups of $\Z^m$ are considered. We begin by summarizing some generalities on semigroups that will be useful for the understanding of this work. Let $\mS := \langle a_1,\ldots,a_n \rangle \subset \Z^m$ denote the subsemigroup of $\Z^m$ generated by $a_1,\ldots,a_n \in \Z^m$, i.e.,
$$
\mS := \langle a_1,\ldots,a_n \rangle = \{ x_1 a_1 + \cdots + x_n a_n \, \vert \, x_1,\ldots,x_n \in \N\}.
$$
The semigroup $\mS$ induces the binary relation $\lS$ on $\Z^m$ given by
$$
 x \lS y\ \Longleftrightarrow \ y - x \in \mS.
$$
It turns out that $(\Z^m, \lS)$ is a poset if and only if $\mS$ is pointed. Indeed, $\lS$ is antisymmetric if and only if $\mS$ is pointed. Moreover, if $\mS$ is pointed then the poset $(\Z^m, \lS)$ is locally finite.

Let $\mu_{\mS}$ denote the M\"obius function associated to $(\Z^m,\lS)$. It is easy to see that $\mu_{\mS}$ can be considered as a univariate function of $\Z^m$. Indeed, for all $x,y\in\Z^m$ and for all $l \ge 0$, one can observe that $c_l(x,y)=c_l(0,y-x)$. Thus, we obtain
$$
\mu_{\mS}(x,y)=\mu_{\mS}(0,y-x)
$$
for all $x,y\in\Z$.

In the sequel of this paper, we shall only consider the reduced M\"obius function $\mu_{\mS} : \Z^m \longrightarrow \Z$ defined by
$$
\mu_{\mS}(x) := \mu_{\mS}(0,x),\text{ for all}\ x\in\Z^m.
$$
Thus, the formula given by (\ref{eq2}) may now be simplified when the locally finite poset is $(\Z^m,\lS)$.

\begin{proposition}\label{in} (\cite[Proposition 1]{CA})
Let $\mS$ be a pointed semigroup and let $x\in\Z^m$. Then,
$$
\sum_{b \in \mS}\mu_{\mS}(x - b) = \left\{\begin{array}{ll}
 1 & \text{if } x = 0,\\
 0 & \text{otherwise}.
\end{array}\right.
$$
\end{proposition}

\begin{proof}
From (\ref{eq1}), we know that  $\mu_{\mS}(b) = 0$ for all $b\notin \mS$. Since $\mS$ is pointed, it follows that
$$
\sum_{b \in \mS}\mu_{\mS}(0 -b) = \mu_{\mS}(0) = 1.
$$
Finally, if $x \neq 0$, then we apply (\ref{eq2}) and we obtain that
$$
0 = \sum_{b \in [0,x]_{\Z^m}} \mu_{\mS}(b) = \sum_{b \in \mS \atop x - b \in \mS} \mu_{\mS}(b) =  \sum_{b \in \mS \atop x - b \in \mS} \mu_{\mS}(x - b) = \sum_{b \in \mS} \mu_{\mS}(x-b).
$$
\end{proof}

\noindent Proposition \ref{in} will be very useful to obtain most of our results.

\section{The Hilbert and M\"obius series}\label{secHilbert}

In this section, we present two results (Theorem \ref{chido} and Theorem \ref{genfunmob}), both relating the Hilbert series of the semigroup $\mS$ with the M\"obius function of the poset $(\Z^m, \leq_{\mS})$. Before proving these theorems, some basic notions on multivariate Hilbert series are quickly recalled. For a thorough study of multivariate Hilbert series, we refer the reader to \cite{KR}.

Let $k$ be any field and let $\mS = \langle a_1,\ldots,a_n \rangle$ be a subsemigroup of $\Z^m$. The semigroup $\mS$ induces a grading on the ring of polynomials $R := k[x_1,\ldots,x_n]$ by assigning ${\rm deg}_{\mS}(x_i) := a_i \in \Z^m$, for all $i \in \{1,\ldots,n\}$. Then, the \emph{$\mS$-degree} of the monomial $m := x_1^{\alpha_1} \cdots x_n^{\alpha_n}$ is ${\rm deg}_{\mS}(m) := \sum_{i=1}^{n}\alpha_i a_i \in \Z^m$. A polynomial is said to be \emph{$\mS$-homogeneous} if all of its monomials have the same $\mS$-degree and an ideal is \emph{$\mS$-homogeneous} if it is generated by $\mS$-homogeneous polynomials. For all $b \in \Z^m$, we denote by $R_b$ the $k$-vector space generated by all $\mS$-homogeneous polynomials of $\mS$-degree $b$.

Whenever $\mS$ is pointed, the $k$-vector space $R_b$ has finite dimension, for all $b \in \Z^m$ (see \cite[Proposition 4.1.19]{KR}). Let $I \subset R$ be an $\mS$-homogeneous ideal. The \emph{multigraded Hilbert function} of $M := R / I$ is
$$
H\!F_{M}: \Z^m \longrightarrow \N,
$$
defined by $H\!F_{M}(b) :=  {\rm dim}_k(R_{b}) - {\rm dim}_k(R_{b} \cap I)$, for all $b \in \Z^m$.

For every $b = (b_1,\ldots,b_m) \in \Z^m$, we denote by $\t^b$ the monomial $t_1^{b_1}\cdots t_m^{b_m}$ in the Laurent polynomial ring $\Z[t_1,\ldots,t_m, t_1^{-1},\ldots,t_m^{-1}]$. The \emph{multivariate Hilbert series of $M$} is the following formal power series in $\Z[[t_1,\ldots,t_m,t_1^{-1},\ldots,t_m^{-1}]]$:
$$
\mH_{M}(\t) := \sum_{b \in \Z^m} H\!F_{M}(b)\, \t^{b}.
$$

We denote by $I_{\mS}$ the \emph{toric ideal of $\mS$}, i.e., the kernel of the homomorphism of $k$-algebras
$$
\varphi: R \longrightarrow k[t_1,\ldots,t_m,t_1^{-1},\ldots,t_m^{-1}]
$$
induced by $\varphi(x_i) = \t^{a_i}$, for all $i \in \{1,\ldots,n\}$. It is well known that $I_{\mS}$ is $\mS$-homogeneous (see \cite[Corollary~4.3]{Sturm}). Moreover, the multivariate Hilbert series of $M = R/I_{\mS}$ with respect to the grading induced by $\mS$ is
\begin{equation}\label{series}
\mathcal H_{M}(\t) = \sum_{b \in \mS} \t^b.
\end{equation}
Indeed, $R_b = \{0\}$ and $H\!F_M(b) = 0$ whenever $b \notin \mS$. In addition, if $b \in \mS$, $\varphi$ induces an isomorphism of $k$-vector spaces between $R_b / (R_b \cap I)$ and $\{\alpha\, \t^b\, \vert \, \alpha \in k\}$, for all $b \in \mS$. Hence, $H\!F_{M}(b) =  1$ in this case.

From now on, the multivariate Hilbert series of $R / I_{\mS}$ is called the \emph{Hilbert series of $\mS$} and is denoted by $\mH_{\mS}(\t)$.

\begin{theorem}\label{chido}
Let $\mS$ be a pointed semigroup and let $c_1,\ldots,c_k$ be nonzero vectors in\, $\Z^m$. If we set
$$
\left(1-\t^{c_1}\right) \cdots \left(1 - \t^{c_k}\right)\ \mH_{\mS}(\t)  = \sum_{b \in \Z^m} f_b \, \t^b \in \Z[[t_1,\ldots,t_m,t_1^{-1},\ldots,t_m^{-1}]],
$$
then,
$$
\sum_{b \in \Z^m} f_b\, \mu_{\mS}(x - b) = 0
$$
for all $x \notin \left\{ \sum_{i \in A} c_i \ \middle|\  A \subset \{1,\ldots,k\} \right\}$.
\end{theorem}

\begin{proof}
From (\ref{series}), we know that
$$
f_b = \sum_{A \subset\{1,\ldots,k\} \atop b - \sum_{i \in A} c_i \in \mS} (-1)^{\vert A\vert},
$$
for all $b \in \Z^m$. Set $\Delta := \left\{ \sum_{i \in A}c_i \ \middle| \ A \subset \{1,\ldots,k\} \right\}$. By Proposition~\ref{in}, for all $x \notin \Delta$ and $A \subset \{1,\ldots,k\}$, we have that
$$
\sum_{b \in \mS} \mu_{\mS}\left(x - \sum_{i \in A} a_i - b\right) = 0.
$$
Hence, for all $x \notin \Delta$, it follows that
$$
\sum_{b \in \Z^m}  \alpha_b\, \mu_{\mS}(x-b) = \sum_{b \in \mS} \sum_{A \subset \{1,\ldots,k\}} (-1)^{\vert A \vert} \mu_{\mS}\left(x - \sum_{i\in A} c_i - b\right) = 0,
$$
where
$$
\alpha_b = \sum_{A \subset\{1,\ldots,k\} \atop b - \sum_{i \in A} c_i \in \mS} (-1)^{\vert A \vert} = f_b.
$$
This completes the proof.
\end{proof}

Notice that the formula $\left(1-\t^{c_1}\right) \cdots \left(1 - \t^{c_k}\right)\ \mH_{\mS}(\t) = \sum_{b \in \Z^m} f_b \, \t^b$ might have an infinite number of terms. Nevertheless, for every $x \in \Z^m$, the formula $\sum_{b \in \Z^m} f_b\, \mu_{\mS}(x - b) = 0$ only involves a finite number of nonzero summands, since $\mS$ is pointed.

The following example illustrates how to apply Theorem \ref{chido} to compute $\mu_{\mS}$.

\begin{example}
Consider the semigroup $\mS = \langle 2, 3\rangle\subset \N$. We observe that $\mS = \N \setminus \{1\}$. Hence, $\mH_{\mS}(t) = 1 + \sum_{b \geq 2} t^b \in \Z[[t]]$ and $t^2 \,\mH_{\mS}(t) = t^2 + \sum_{b \geq 4} t^b$. It follows that
$$
(1 - t^2)\, \mH_{\mS}(t) = 1 + t^3.
$$
Applying Theorem \ref{chido}, we get that
$$
\mu_{\mS}(x) + \mu_{\mS}(x-3) = 0,
$$
for all $x \in \Z \setminus \{0,2\}$. Furthermore, by direct computation, we have $\mu_{\mS}(0) = 1$, $\mu_{\mS}(2) = -1$ and $\mu_{\mS}(x) = 0$ for all $x < 0$. This leads to the formula
$$
\mu_{\mS}(x) = \left\{
\begin{array}{rl}
1 & \text{if } x \geq 0 \text{ and } x \equiv 0 \text{ or } 5 \pmod{6},\\
-1 & \text{if } x \geq 0 \text{ and } x \equiv 2 \text{ or } 3 \pmod{6},\\
0 & \text{otherwise.}
\end{array} \right.
$$
\end{example}

From now on, we consider the {\it M\"obius series} $\mG_{\mS}$, i.e., the generating function of the M\"{o}bius function
$$
\mG_{\mS}(\t) := \sum_{b \in \Z^m} \mu_{\mS}(b)\, \t^b \in \Z[[t_1,\ldots,t_m,t_1^{-1},\ldots,t_m^{-1}]].
$$

\begin{theorem}\label{genfunmob}
Let $\mS$ be a pointed semigroup. Then,
$$
\mH_{\mS}(\t)\cdot \mG_{\mS}(\t)  = 1.
$$
\end{theorem}

\begin{proof}
From the definitions of $\mH_{\mS}(\t)$ and $\mG_{\mS}(\t)$, we obtain that
$$
\mH_{\mS}(\t)\cdot \mG_{\mS}(\t) = \left(\sum_{b \in \mS}  \t^b\right)\left(\sum_{b\in\Z^m} \mu_{\mS}(b) \t^b\right) = \sum_{b \in \Z^m} \left(\sum_{c \in \mS}\mu_{\mS}(b-c)\right) \t^b.
$$
The result follows by Proposition~\ref{in}.
\end{proof}

Theorem~\ref{genfunmob} states that, whenever we can explicitly compute the inverse of $\mH_{\mS}(\t)$, we will be able to obtain $\mu_{\mS}$. We illustrate this idea in our next example.

\begin{example}\label{nm}
Let $\{e_1,\ldots,e_m\}$ denote the canonical basis of $\N^m$ and let $\mS = \langle e_1,\ldots,e_m \rangle = \N^m$. Clearly, we have that
$$
\mH_{\N^m}(\t) = \sum_{b \in\N^m}  \t^b = \frac{1}{(1-t_1) \cdots (1- t_m)}.
$$
Therefore, by Theorem~\ref{genfunmob}, we obtain
$$
\mG_{\N^m}(\t) = \left(1- t_1\right) \cdots \left(1- t_m\right) = \sum_{A \subset\{1,\ldots,m\}} (-1)^{\mid A \mid}\, \prod_{i \in A}\, t_i = \sum_{A\subset \{1,\ldots,m\}} (-1)^{\mid A \mid}\, \t^{\sum_{i \in A}e_i}.
$$
So we derive the following formula for $\mu_{\N^m}$:
$$
\mu_{\N^m}(x) = \left\{ \begin{array}{cl}
(-1)^{\vert A \vert} & {\text \ if \ } x = \sum_{i \in A} e_i \text{ for some } A \subset \{1,\ldots,m\}, \\ \ \\
0 & \text{otherwise}.
\end{array}\right.
$$
\end{example}

A pointed semigroup $\mS = \langle a_1,\ldots, a_n \rangle$ is called a \emph{complete intersection} semigroup if its corresponding toric ideal $\IS$ is a \emph{complete intersection} ideal, i.e., if $\IS$ is generated by $n-d$ $\mS$-homogeneous polynomials, where $d$ is the dimension of the $\Q$-vector space spanned by $a_1,\ldots,a_n$. For characterizations of complete intersection toric ideals, we refer the reader to \cite{FMS}.

Let $B = (b_1,b_2,\ldots,b_k)$ be a $k$-tuple of nonzero vectors in $\Z^m$ such that the semigroup $\mathcal T := \langle b_1,\ldots,b_k\rangle$ is pointed and let $b \in \Z^m$. We denote by $d_B(b)$ the number of non-negative integer representations of $b$ by $b_1,\ldots,b_k$, that is, the number of solutions of $b =\sum_{i=1}^{k} x_i b_i$, where $x_i$ is a nonnegative integer for all $i$. Since $\mathcal T$ is pointed, we know that $d_B(b)$ is finite, for all $b \in \Z^m$. Moreover, $d_B(0)=1$. It is well known (see, e.g., \cite[Theorem 5.8.15]{KR}) that its generating function is given by
$$
\sum_{b \in\Z^m} d_B(b)\, \t^b = \frac{1}{\left(1-\t^{b_1}\right)\left(1-\t^{b_2}\right)\cdots \left(1-\t^{b_k}\right)}.
$$

\begin{corollary}\label{mobden}
Let $\mS$ be a complete intersection pointed semigroup and assume that $\IS$ is generated by $n-d$ $\mS$-homogeneous polynomials of $\mS$-degrees $b_1,\ldots,b_{n-d} \in \Z^m$. Then,
$$
\mu_{\mS}(x) = \sum_{A \subset \{1,\ldots,n\}} (-1)^{\mid A \mid} \, d_B\left(x - \sum_{i \in A} b_i\right),
$$
for all $x \in \Z^m$, where $B = \{b_1,\ldots,b_{n-d}\}$.
\end{corollary}

\begin{proof}
By \cite[Page 341]{KR}, we have that
$$
\mH_{\mS}(\t) = \frac{ (1 - \t^{b_1}) \cdots (1 - \t^{b_{n-d}}) }{(1 - \t^{a_1}) \cdots (1 - \t^{a_n})}.
$$
Thus, from Theorem~\ref{genfunmob}, we obtain
$$
\mG_{\mS}(\t)
\begin{array}[t]{l}
 = \displaystyle\frac{1}{\mH_{\mS}(\t)} = \frac{ (1 - \t^{a_1}) \cdots (1 - \t^{a_n})}{{(1-\t^{b_1})\cdots (1-\t^{b_{n-d}})}} \\ \ \\
 = \displaystyle\left(\sum_{A \subset \{1,\ldots,n\}} (-1)^{\vert A \vert}\, \t^{\sum_{i \in A} a_i}\right) \left( \sum_{b \in \Z^m} d_B(b)
\, \t^b\right) \\ \ \\  = \displaystyle \sum_{b \in \Z^m} \sum_{A
\subset \{1,\ldots,n\}} (-1)^{\vert A \vert}\,
  d_B(b)\, \t^{b + \sum_{i \in A} a_i} \\ \ \\ = \displaystyle \sum_{b \in \Z^m} \sum_{A \subset \{1,\ldots,n\}} (-1)^{\vert A \vert}\,
  d_B\left(b -  \sum_{i \in A} a_i\right)\, \t^{b}.
\end{array}
$$
\end{proof}

\section{Explicit formulas for the M\"obius function}

In this section, we exploit the results of the previous section to obtain explicit formulas for $\mu_{\mS}$ when $\mS$ is a semigroup with a unique Betti element (Theorem \ref{uniquebetti}) and when $\mS$ is a complete intersection numerical semigroup generated by three elements (Theorem \ref{3ic}).

The results included in this section are consequences of Corollary~\ref{mobden}. However, they can also be obtained with a different proof by using Theorem~\ref{chido}.

\subsection{Semigroups with a unique Betti element.}

A semigroup $\mS \subset \N^m$ is said to have a \emph{unique Betti element} $b \in \N^m$ if its corresponding toric ideal is generated by a set of $\mS$-homogeneous polynomials of common $\mS$-degree $b$. Garc\'{i}a-S\'{a}nchez, Ojeda and Rosales proved \cite[Corollary 10]{GOR} that these semigroups are always complete intersection.

\begin{theorem}\label{uniquebetti}
Let $\mS = \langle a_1,\ldots, a_n \rangle \subset \N^m$ be a semigroup with a unique Betti element $b \in \N^m$. If we denote by $d$ the dimension of the $\Q$-vector space generated by $a_1,\ldots,a_n$, then we have
$$
\mu_{\mS}(x) = \sum_{j = 1}^t \, (-1)^{\mid A_j \mid} \, {{k_{A_j} + n - d - 1} \choose {k_{A_j}}},
$$
where $\{A_1,\ldots,A_t\} = \left\{A \subset \{1,\ldots,n\}\ \middle| \text{ there exists }  k_{A} \in \N \text{ such that } x - \sum_{i \in A} a_i = k_A\, b\right\}$.
\end{theorem}

\begin{proof}
By Corollary~\ref{mobden}, for all $x \in \Z^m$, we have
$$
\mu_{\mS}(x) = \sum_{A \subset \{1,\ldots,m\}} (-1)^{\mid A \mid} \, d_B\left(x - \sum_{i \in A} a_i\right),
$$
where $B$ is the $(n-d)$-tuple $(b,\ldots,b)$. The equality
$$
d_B(y) = \left\{ \begin{array}{cl}
\displaystyle\binom{k + n - d - 1}{k} & \text{ if } y = k b \text{ with } k \in \N ,\\ \ \\
0 & \text{otherwise},
\end{array}\right.
$$
for all $y \in \Z^m$, completes the proof.
\end{proof}

When $m = 1$, i.e., when $\mS = \langle a_1,\ldots,a_n \rangle\subset \N$, $\mS$ is a numerical semigroup with a unique Betti element $b \in \N$ if and only if there exist pairwise relatively prime integers $b_1,\ldots,b_n \geq 2$ such that $a_i := \prod_{j\neq i} b_j$, for all $i \in \{1,\ldots,n\}$, and $b = \prod_{i =1}^n b_i$ (see \cite{GOR}). In this setting, Theorem~\ref{uniquebetti} can be refined as follows.

\begin{corollary}\label{uniquebettinumerical}
Let $\mS = \langle a_1,\ldots, a_n \rangle \subset \N$ be a numerical semigroup with a unique Betti element $b \in \N$. Then,
$$
\mu_{\mS}(x) = \left\{ \begin{array}{cl}
\displaystyle(-1)^{\mid A \mid} \, {\binom{k + n - 2}{k}}&\text{if } x = \sum_{i \in A} a_i + k b  \text{ for some } A \subset \{1,\ldots,n\}, k \in \N,\\ \ \\
0 & \text{otherwise.}
\end{array}\right.
$$
\end{corollary}

\begin{proof}
Since $d = 1$, it is sufficient, by Theorem~\ref{uniquebetti}, to prove that, for every $A_1,A_2\subset \{1,\ldots,n\}$, if $b$ divides $\sum_{i\in A_1} a_i -\sum_{i \in A_2} a_i$, then $A_1 = A_2$. Let $b_1,\ldots,b_n \geq 2$ such that $a_i =\prod_{j \neq i} b_j$. By \cite[Example 12]{GOR} we have that $\IS = (f_2,\ldots,f_n)$, where $f_i := x_1^{b_1} - x_i^{b_i}$ for all $i \in \{2,\ldots,n\}$. Assume that there exist $A_1, A_2\subset \{1,\ldots,n\}$ such that $A_1 \not=A_2$ and $\sum_{i \in A_1}a_i - \sum_{i \in A_2} a_i = k b$, for some $k \in \N$. Thus, the binomial $g := \prod_{i \in A_1} x_i -x_1^{b_1 k}\, \prod_{i \in A_2} x_i \not= 0$ belongs to $\IS$ and it can be written as a combination of $f_2,\ldots,f_n$. However, since $x_j^{b_j}$ does not divide $\prod_{i \in A_1} x_i$ for all $j \in \{1,\ldots,n\}$, we obtain a contradiction.
\end{proof}

As a direct consequence of this result, we recover Dedden's result.

\begin{corollary}\cite{Deddens}
Let $a,b \in \Z^+$ be relatively prime integers and consider $\mS := \langle a, b \rangle\subset\N$. Then,
$$
\mu_{\mS}(x) = \left\{ \begin{array}{rl}
 1 & {\text \  if \ } x \geq 0 \ {\text and \ } x \equiv 0 {\text \ or \ } a+b \ ({\rm mod \ } ab),\\
 -1 & {\text \ if \ } x \geq 0 {\text \ and \ } x \equiv a {\text \ or \ } b \ ({\text mod \ } ab),\\
 0 & {\text \ otherwise.}
\end{array}\right.
$$
\end{corollary}

\subsection{Three generated complete intersection numerical semigroups.}

We provide a semi-explicit formula for $\mu_{\mS}$, when $\mS$ is a complete intersection numerical semigroup minimally generated by the set $\{a_1, a_2, a_3\}$. When $\mS = \langle a_1,a_2,a_3 \rangle \subset \N$, Herzog proves in \cite{Herzog} that $\mS$ is a complete intersection if and only if $\gcd\{a_i,a_j\}\, a_k \in \langle a_i, a_j \rangle$ with $\{i,j,k\} = \{1,2,3\}$. Suppose that $da_1 \in \langle a_2, a_3 \rangle$, where $d := \gcd\{a_2,a_3\}$.

For every $x \in \Z$, there exists a unique $\alpha(x) \in\{0,\ldots,d-1\}$ such that $\alpha(x) a_1 \equiv x \ ({\rm mod}\ d)$. It is easy to check that, for every $x, y \in \Z$,
\begin{equation}\label{alfa}
\alpha(x - y) = \left\{ \begin{array}{cl}
\alpha(x) - \alpha(y) & \text{if } \alpha(x)\geq \alpha(y),\\ \\
d + \alpha(x) - \alpha(y) & \text{otherwise}.
\end{array}\right.
\end{equation}

\begin{theorem}\label{3ic}
Let $\mS = \langle a_1,a_2,a_3 \rangle$ be a numerical semigroup such that $d a_1 \in \langle a_2,a_3 \rangle$, where $d :=
{\gcd}\{a_2,a_3\}$. For all $x  \in \Z$, we have that $\mu_{\mS}(x) = 0 $, if $\alpha(x) \geq 2$, and
$$
\mu_{\mS}(x) = (-1)^{\alpha}\, \left(d_B(x') - d_B(x'-a_2) - d_B(x'-a_3) + d_B(x' - a_2 - a_3)\right)
$$
otherwise, where $x' := x - \alpha(x) a_1$ and $B := (d a_1, a_2\, a_3 / d)$.
\end{theorem}

\begin{proof}
Suppose that $d a_1 = \gamma_2 a_2 + \gamma_3 a_3$ with $\gamma_2,\gamma_3 \in \N$. Then, by \cite[Theorem 3.10]{Herzog}, it follows that
$$
I_{\mS} = \left(x_1^d - x_2^{\gamma_2} x_3^{\gamma_3},\, x_2^{a_3/d} -x_3^{a_2/d}\right).
$$
So, $I_{\mS}$ is generated by two $\mS$-homogeneous polynomials of $\mS$-degrees $d a_1$ and $a_2 a_3 / d$. Hence, from Corollary~\ref{mobden}, we have
\begin{equation}\label{mudenum}
\mu_{\mS}(x) =
\begin{array}[t]{l}
 d_{B}(x) - d_{B}(x-a_1) - d_{B}(x-a_2) - d_{B}(x-a_3) + d_{B}(x-(a_1+a_2)) + \\
 + d_{B}(x-(a_1+a_3)) + d_{B}(x-(a_2+a_3)) - d_{B}(x-(a_1+a_2+a_3)), \\
\end{array}
\end{equation}
for all integers $x$, where $B :=(d a_1, a_2 a_3 / d)$. Since $\alpha(d a_1) = \alpha(a_2 a_3 / d) = 0$. It follows that $\alpha(y) = 0$ if $y \in \langle d a_1, a_2 a_3 / d \rangle$. As a consequence of this, $d_B(y) = 0$ whenever $\alpha(y) \neq 0$.

Let $C:=\{0,a_1,a_2,a_3,a_1+a_2,a_2+a_3,a_3+a_1,a_1+a_2+a_3\}$. Notice that $\alpha(y) \in \{0,1\}$, for all $y \in C$. We distinguish three different cases upon the value of $\alpha :=\alpha(x)$, for $x \in\Z$.
\begin{case}
$\alpha \geq 2$.\\
We deduce that $\alpha(x - y) = \alpha(x) - \alpha(y) \neq 0$ and $d_B(x - y)=0$, for all $y\in C$. Therefore, using (\ref{mudenum}), we obtain that $\mu_{\mS}(x)=0$.
\end{case}
\begin{case}
$\alpha = 1$.\\
We deduce that $\alpha(x - y) \neq 0$ and $d_B(x-y) = 0$ for all $y\in \{0,a_2,a_3,a_2+a_3\}$. Therefore, using (\ref{mudenum}), we obtain that
$$
\mu_{\mS}(x) = -d_B(x - a_1) + d_B(x-a_1-a_2) + d_B(x-a_1-a_3) -
d_B(x-a_1-a_2-a_3).
$$
\end{case}
\begin{case}
$\alpha = 0$.\\
Since $d \geq 2$, we deduce that $\alpha(x - y) \neq 0$ and $d_B(x-y) = 0$ for all $y\in \{a_1,a_1+a_2,a_1+a_3,a_1+a_2+a_3\}$. Therefore, using (\ref{mudenum}), we obtain that
$$
\mu_{\mS}(x) = d_B(x) - d_B(x-a_2) - d_B(x-a_3) + d_B(x-a_2-a_3).
$$
\end{case}
\noindent This completes the proof.
\end{proof}

Theorem~\ref{3ic} yields an algorithm for computing $\mu_{\mS}(x)$, for all $x \in \Z$, which relies on the computation of four values of $d_{B}(y)$, where $B = (d a_1, a_2 a_3 / d)$. It is worth mentioning that in \cite[Section 4.4]{Alfonsin} there are several results and methods to compute these values.

Also note that Theorem~\ref{3ic} generalizes \cite[Theorem 3]{CA}, where the authors provide a semi-explicit formula for $\mS = \langle 2q, 2q + e, 2q + 2e \rangle$ where $q, e \in \Z^+$ and $\gcd\{2q,2q+e,2q+2e\} = 1$. Indeed, if $\mS = \langle a, a + e, \ldots, a + ke \rangle$ with $\gcd\{a,e\} = 1$ and $k \geq 2$, then $\mS$ is a complete intersection if and only if $k = 2$ and $a$ is even (see \cite{BG}).

\section{When is a poset equivalent to a semigroup poset?}

A natural question is whether a poset $\mP$ is {\em isomorphic} to a poset associated to a semigroup $\mS$ since, in such a case, one might be able to calculate $\mu_\mP$  by computing $\mu_\mS$ instead. Let us illustrate this with the following two examples in which we can easily find an appropriate order isomorphism 
between the poset $\mP$ and the poset associated to the semigroup $\N^m = \langle e_1,\ldots,e_m \rangle$.

\begin{example} We consider the classical arithmetic M\"obius function $\mu$. Recall that for all $a,b \in \N$ such that $a \mid b$, we have that
\begin{equation}\label{aritm}
\mu (a,b) = \left\{\begin{array}{cl}
 (-1)^{r} &  \text{if } b / a \text{ is a product of } r \text{ different prime numbers,} \\
  0 & \text{otherwise}.
\end{array}\right.
\end{equation}
For every $m \in \Z^+$, we denote by $p_1,\ldots, p_m$ the first $m$ prime numbers and by $\N_m$ the set of integers that can be written as a product of powers of $p_1,\ldots,p_m$. Then, for all $m \geq 1$, the map $\psi: \N_m \rightarrow \N^m$ defined as $\psi(p_1^{\,\alpha_1} \cdots p_m^{\,\alpha_m}) = (\alpha_1,\ldots,\alpha_m)$ is an order isomorphism between $\N_m$, ordered by divisibility, and the poset $(\N^m,\leq_{\N^m})$. Hence, for every $a, b \in \N$, we consider $m \in \Z^+$ such that $a, b \in \N_m$ and we recover the formula (\ref{aritm}) by means of the M\"obius function of $\N^m$ given in Example~\ref{nm}.
\end{example}
 
\begin{example}
Let $D = \{d_1,\ldots,d_m\}$ be a finite set and let us consider the (locally finite) poset $\mP$ of multisets of $D$ ordered by inclusion. For every $S, T \in \mP$ such that $T \subset S$, it is well known that
\begin{equation}\label{musubset}
\mu_{\mP} (T,S) = \left\{\begin{array}{cl}
 (-1)^{|S \setminus T|} &  \text{if } T \subset S \text{ and }S \setminus T \text{ is a set,} \\
 0 & \text{otherwise.}
\end{array}\right.
\end{equation}
\end{example}

We consider the map $\psi: \mP \rightarrow \N^m$ defined as $\psi(S) = (s_1,\ldots,s_m)$, where $s_i$ denotes the multiplicity of $d_i$ in $S$, for all $S \in \mP$. We consider the order in $\N^m$ induced by the semigroup $\N^m$, i.e., $\alpha \leq_{\N^m} \beta$ if and only if $\beta - \alpha \in \N^m$ for all $\alpha, \beta \in \N^m$. We have that $\psi$ is an {\it order isomorphism}, i.e., an order preserving and order reflecting bijection. Thus, we can say that the poset of multisets of a finite set is a particular case of semigroup poset. This implies that for all $S, T \in \mP$ such that $T \subset S$, $\mu_{\mP}(T,S) = \mu_{\N^m}(\psi(T), \psi(S)) = \mu_{\N^m}(\psi(S) - \psi(T))$ and by Example~\ref{nm} we retrieve the formula (\ref{musubset}).

In the rest of this section, we present a characterization of those locally finite posets $\mP$ isomorphic to the poset associated to a semigroup $\mS$ (Theorem~\ref{equivasemigrupo}).

Let $(\mP, \leq_{\mP})$ be a locally finite poset. For every $x \in\mP$, we set $\mP_x := \{y \in \mP \, \vert \, x \leq_{\mP} y\}$ and we consider the restricted M\"obius function $\mu_{\mP}(-,x): \mP_x\rightarrow \Z$. It is clear that, if there exists a pointed semigroup $\mS$ and an order isomorphism $\psi: (\mP_x, \leq_{\mP}) \longrightarrow (\mS, \leq_{\mS})$, then $\mu_{\mP}(-,x)$ can be computed by means of the M\"obius function of $(\mS,\leq_{\mS})$, since $\mu_{\mP}(y,x) = \mu_{\mS}(\psi(y))$ for all $y \in \mP_x$. 

The poset $\mP_x$ is said to be \emph{autoequivalent} if and only if, for all $y \in \mP_x$, there exists an order isomorphism $g_y: \mP_x \longrightarrow \mP_y$ such that $g_y \circ g_z = g_z \circ g_y$, for all $y,z \in \mP_x$, and $g_x$ is the identity. For all $y \in \mP_x$, we set  $l_1(y) := \{z \in \mP \, \vert\, \text{ there is no } u \in \mP$ such that $y \lneq u \lneq z\}$. Whenever $\mP_x$ is autoequivalent with isomorphisms $\{g_y\}_{x \leq y}$ and $l_1(x)$ is a finite set of $n$ elements, we  associate to $\mP$ a subgroup $L_{\mP} \subset \Z^n$ in the following way.

Let $l_1(x) = \{x_1,\ldots,x_n\} \subset \mP$ and consider the map
$$
f: \N^n \longrightarrow \mP
$$
defined as $f(0,\ldots,0) = x$, and for all $\alpha \in \N^n$ and all $i \in \{1,\ldots,n\}$,  $f(\alpha + e_i) = g_{x_i}(f(\alpha))$, where $\{e_1,\ldots,e_n\}$ is the canonical basis of $\N^m$. In particular, $f(e_i) = g_{x_i}(f(0)) = g_{x_i}(x) = x_i$, for all $i \in \{1,\ldots,n\}$.

\begin{lemma}\label{surjective}
$f$ is well defined and is surjective.
\end{lemma}

\begin{proof}
Suppose that $\alpha + e_i = \beta + e_j$. Then, we set $\gamma := \alpha - e_j = \beta - e_i \in \N^n$. Thus, $f(\alpha + e_i) = g_{x_i} (f(\alpha)) = g_{x_i}(g_{x_j}(f(\gamma))) = g_{x_j}(g_{x_i}(f(\gamma))) = g_{x_j} (f(\beta)) = f(\beta + e_j)$ and $f$ is well defined.

Take $y \in \mP_x$. If $y = x$, then $y = f(0)$. If $y \neq x$, then there exists $z \in \mP_x$ such that $y \in l_1(z)$. Therefore $y = g_z(x_j)$ for some $j \in \{1,\ldots,n\}$. We claim that if $z = f(\alpha)$, then $y = f(\alpha + e_j)$. Indeed, $f(\alpha + e_j) = g_{x_j}(f(\alpha)) = g_{x_j}(z) = g_{x_j} (g_z(x)) = g_z(g_{x_j}(x)) = g_z(x_j) = y$.
\end{proof}

Now, we set $L_{\mP} := \{\alpha - \beta  \in \Z^n \, \vert \, f(\alpha) = f(\beta) \}$.

\begin{lemma}
$L_{\mP}$ is a subgroup of $\Z^n$.
\end{lemma}

\begin{proof}
If $\gamma \in L_{\mP}$, then $-\gamma \in L_{\mP}$. Moreover, if $\gamma_1, \gamma_2 \in L_{\mP}$, then $\gamma_1 + \gamma_2 \in L_{\mP}$. Indeed, take $\alpha, \alpha', \beta, \beta' \in \N^m$ such that $f(\alpha) = f(\alpha')$, $\gamma_1 = \alpha-\alpha'$, $f(\beta) = f(\beta')$  and $\gamma_2 = \beta - \beta'$. Then $f(\alpha + \beta) = f(\alpha' + \beta) = f(\alpha' + \beta')$ and the lemma is proved.
\end{proof}

If $L$ is a subgroup of $\Z^n$, then its {\it saturation} is the group defined by
$$
\textrm{Sat}(L) := \left\{\gamma \in \Z^n \ \middle| \text{ there exists } d \in \Z^+ \text{ such that } d \gamma \in L\right\}.
$$

\begin{theorem}\label{equivasemigrupo}
Let $\mP$ be a locally finite poset and let $x \in \mP$. Then, $(\mP_x, \leq)$ is isomorphic to $(\mS, \leq_{\mS})$ for some (pointed) semigroup $\mS \subset \Z^{m}$ if and only if $\mP_x$ is autoequivalent, $l_1(x)$ is finite and $L_{\mP} = {\rm Sat}(L_{\mP})$.
\end{theorem}

\begin{proof}
$(\Rightarrow)$ Let $\mS \subset \Z^m$ be a (pointed) semigroup and denote by $\{a_1,\ldots,a_n\}$ its unique minimal set of generators. Assume that $\psi: \mP_x \rightarrow \mS$ is an order isomorphism. Let us prove that $\mP_x$ is autoequivalent, $\left| l_1(x)\right| = n$ and $L_{\mP} = {\rm Sat}(L_{\mP})$. First, we observe that if $x_i := \psi^{-1}(a_i)$, then $l_1(x) = \{x_1,\ldots,x_n\}$. And thus $\left| l_1(x)\right|=n$. Now, for every $y \in \mP_x$, we set
$$
\begin{array}{cccl}
g_y: & \mP_x & \longrightarrow & \mP_y \\
 & z & \longmapsto & \psi^{-1}(\psi(z) + \psi(y)).
\end{array}
$$
Then it is straightforward to check that $g_y$ is an order isomorphism. Moreover, $g_x$ is the identity map on $\mP_x$ and $g_y \circ g_z = g_z \circ g_y$, for all $y,z \in \mP_x$. And thus $\mP_x$ is autoequivalent.

Let $f: \N^m \rightarrow \mP_x$ be the map associated to $\{g_y\}_{y \leq x}$, i.e., $f(0) = x$ and if $f(\alpha) = y$, then $f(\alpha + e_j) = g_{x_j}(f(\alpha))$. We claim that $\psi(f(\alpha)) = \sum \alpha_i a_i \in \mS$, for all $\alpha = (\alpha_1,\ldots,\alpha_n) \in \N^n$. Indeed, $\psi(f(0)) = \psi(x) = 0$ and if we assume that $\psi(f (\alpha)) = \sum \alpha_i a_i$ for some $\alpha = (\alpha_1,\ldots,\alpha_n) \in \N^m$, then $\psi(f(\alpha + e_j)) = \psi (g_{x_j}(\alpha)) = \psi(z) + \psi(x_j) = \sum \alpha_i a_i + a_j$, as desired.

Since $L_{\mP} \subset {\rm Sat}(L_{\mP})$ by definition, let us prove that ${\rm Sat}(L_{\mP}) \subset L_{\mP}$. We take $\gamma \in {\rm Sat}(L_{\mP})$, then $d \gamma \in L_{\mP}$ for some $d \in \Z^+$. This means that there exist $\alpha, \beta \in \N^n$ such that $f(\alpha) = f(\beta)$ and $d \gamma = \alpha - \beta$. Hence, we have that $\sum \alpha_i a_i = \psi(f(\alpha)) = \psi(f(\beta)) = \sum \beta_i a_i$. This implies that $\sum \gamma_i a_i = 1/d \ (\sum (\alpha_i - \beta_i) a_i) = 0$. Thus, if we take $\alpha', \beta' \in \N^m$ such that $\gamma = \alpha' - \beta'$, then $\psi(f(\alpha')) = \psi(f(\beta'))$ and, whence, $f(\alpha') = f(\beta')$ and $\gamma \in L_{\mP}$. And thus $L_{\mP} = {\rm Sat}(L_{\mP})$.

\medskip

$(\Leftarrow)$ Since $L_{\mP} = {\rm Sat}(L_{\mP})$, we have that $\Z^n / L_{\mP}$ is a torsion free group. Hence there exists a group isomorphism $\rho: \Z^n / L_{\mP} \rightarrow \Z^m$, where $m = n - {\rm rk}(L_{\mP})$. We let $a_i := \rho(e_i + L_{\mP})$ for all $i \in \{1,\ldots,n\}$ and set $\mS := \langle a_1,\ldots,a_n\rangle \subset \Z^m$. We claim that $(\mP_x, \leq)$ and $(\mS, \leq_{\mS})$ are isomorphic. More precisely, it is straightforward to check that the map
$$
\begin{array}{cccl}
\psi: & \mP_x & \longrightarrow & \mS \\
& y & \longmapsto & \sum \alpha_i a_i, {\text \ if \ } f(\alpha) = y
\end{array}
$$
is an order isomorphism.
\end{proof}

The necessity direction of Theorem~\ref{equivasemigrupo} can be stated in algebraic terms as : whenever $\mP_x$ is autoequivalent and $l_1(x)$ is finite, the subgroup $L_{\mP}$ defines a lattice ideal $I := (\{ \x^{\alpha} - \x^{\beta} \, \vert \, \alpha - \beta \in L_{\mP}\})$. Moreover, $\mP_x$ is isomorphic to a semigroup poset $(\mS, \lS)$ if and only if the ideal $I$ itself is the toric ideal of a semigroup $\mS$. The latter holds if and only if $I$ is prime or, equivalently, if $L_{\mP} = {\rm Sat}(L_{\mP})$ (see \cite{ES}).

\section*{Acknowledgments}
The authors would like to thank the anonymous referees for their valuable comments and suggestions.

\bibliographystyle{plain}

\end{document}